\documentclass[reqno,11pt]{amsart}
\usepackage{amssymb,amsfonts,amsbsy,setspace}

\makeatletter

\everymath{\displaystyle} 

 \theoremstyle{plain}
 \newtheorem{thm}{Theorem}[section]
 \newtheorem{lem}[thm]{Lemma}
 
 \numberwithin{equation}{section} 
 \numberwithin{figure}{section} 
 \newtheorem{prop}[thm]{Proposition}
 
 \theoremstyle{remark}
 \newtheorem{rmk}[thm]{Remark}

 \newtheorem*{acknowledgement*}{Acknowledgement}
 \theoremstyle{definition}
 \newtheorem{defi}[thm]{Definition}

\def\eps{\varepsilon}
\def\bgamma{{\overline{\gamma}}}
\def\W{\mathcal W}

\newcommand{\R}{\mathbb R}

\newcommand{\Op}{\mathcal{O}}
\newcommand{\N}{\mathbb N}

\DeclareMathOperator{\supp}{spt}
\def\esssup#1{\underset{#1}{\rm ess.sup}\,}

\newcommand{\M}{\mathcal{M}}
\def\res{\lfloor}
\newcommand{\dx}{\,\text{\normalfont d}}
\newcommand{\icm}{{\boldmath \bf $\infty$-cm\ }}

\def\bx{{\overline{x}}}
\def\by{{\overline{y}}}

\usepackage{enumitem}
\usepackage{color}

\renewcommand{\phi}{\varphi}
\usepackage{hyperref}
\newcommand{\id}{\text{\normalfont id}}
\renewcommand{\neg}{\text{neg}}

\renewcommand{\O}{\mathcal{O}}



\begin{document}

\title[Duality for the one-dimensional $L^\infty$-optimal transport]{A study of the dual problem of the one-dimensional $L^\infty$-optimal transport problem with applications}

\author{Luigi De Pascale}\address{{\bf L.D.P.} Dipartimento di Matematica ed Informatica, 
Universit\'a di Firenze, Viale Morgagni, 67/a - 50134 Firenze, ITALY} 

\email{luigi.depascale@unifi.it}

\author{Jean Louet}\address{{\bf J.L.} CEREMADE, Universit\'e Paris-Dauphine, Place du Mar\'echal de Lattre de Tassigny, 75 775 Paris cedex 16, France; INRIA, MOKAPLAN, 2~rue Simone Iff, 75 012 Paris, France}         

\email{louet@ceremade.dauphine.fr}

\begin{abstract} The Monge-Kantorovich problem for the $\W_\infty$ distance presents several peculiarities. Among them the lack of convexity and then of a direct duality.  
We study in dimension 1 the dual problem introduced by Barron, Bocea and Jensen in~\cite{BarBocJen2017}. We construct a couple of Kantorovich potentials which is non 
trivial in the best possible way. More precisely, we build a potential which is non constant around any point that the restrictable, minimizing plan moves at maximal distance. 
As an application, we show that the set of points which are displaced at maximal distance by a 
``locally optimal'' transport plan is shared by all the other optimal transport plans, and we describe the general structure of all the one-dimensional optimal transport plans.
\end{abstract}

\date{\today}
\maketitle

\section{Introduction}

Given two probability measures $\mu,\nu$ on $\R^d$, the infinite Wasserstein distance between $\mu$ and $\nu$ is defined as 
\begin{equation}\label{KantInf}
\W_\infty (\mu,\nu):=\min_{\gamma \in \Pi(\mu,\nu)} \gamma-\esssup {\R^d\times \R^d} |x-y|,
\end{equation} 
where $\Pi(\mu,\nu)$ denotes the set of positive measures on $\R^d$ whose first and second marginals are $\mu,\nu$ respectively.
This distance is the natural limit as $p \to \infty$ of the more common Kantorovich-Wasserstein distances $\W_p$.

Problem (\ref{KantInf}) above was first studied in \cite{ChaDePJuu2008} where it was observed that, in spite of its proximity with the $\W_p$ distances, 
it presents several peculiarities. 
The two most striking phenomena are the lack of linearity or even convexity of (\ref{KantInf}) with respect to $\gamma$ (in fact the functional $ \gamma \mapsto \|y-x\|_{L^\infty_\gamma(\R^{2d})}$ is only level-convex), and the non-uniqueness of the minimizer: indeed, one may guess that, from any optimal plan~$\gamma$, any small perturbation of~$\gamma$ around a~point $(x,y)$ which is not moved at maximal distance will provide a new and different optimal transport plan. It is therefore necessary to 
find a suitable notion of {\it local solution} of~\eqref{KantInf}. 
The following definition, introduced in~\cite{ChaDePJuu2008}, turned out to be the right one.

\begin{defi} \label{RestSol} A transport plan $\gamma \in \Pi(\mu, \nu)$ is a {\em restrictable solution} of  (\ref{KantInf}) if any positive and nonzero Borel measure  $\gamma'$ on  $\R^{2d}$
that is majorized by $\gamma$ is a solution to the problem
  $$ \inf_{\lambda \in \Pi (\mu',\nu')}  \lambda-\esssup{\R^d\times \R^d} |y-x|,$$
where $\mu'= (\pi_1)_\# \gamma'$ and $\nu'= (\pi_2)_\# \gamma'$.
\end{defi}

It is likely that a small ``local modification'' of a restrictable solution will not be a restrictable solution anymore. 
From the analysis of the problem~\eqref{KantInf} as a limit of the classical optimal transport problem with 
cost $c(x,y)=|y-x|^p$ as $p\to +\infty$, it turns out that the notion of restrictable solution 
is equivalent to the one of {\it infinite cyclical monotonicity} of a transport plan, whose definition is the following:

\begin{defi} \label{DefInfCyclPlan} A transport plan $\gamma \in \Pi (\mu,\nu)$  is 
$\infty$-cyclically-monotone (\icm in the paper) if, 
for all $n \in \N$ and $(x_0, y_0), \dots, (x_n,y_n) \in \supp (\gamma)$, it holds 
\begin{equation} \label{InfCycl} \max\limits_{1\leq i \leq n} \{|x_i-y_i|\} \leq \max\limits_{1\leq i\leq n} \{|x_i-y_{i+1}|\}, \end{equation}
where, as usual, we set $y_{n+1}=y_0$.
\end{defi}

The infinite cyclical monotonicity is the $L^\infty$ version of the $c$-cyclical monotonicity, 
which is a fundamental notion in the classical theory of optimal transportation. 

The equivalence between the two definitions above is proven in~\cite{ChaDePJuu2008}, where it is also proven, under suitable assumptions, 
that \icm transport plans are induced by maps (see Theorems~3.2 and 4.4 therein, and let us also mention~\cite{Jylha} where similar results 
are extended to more general cost functions). 
In fact, using the cyclical monotonicity to prove the existence of an optimal transport map was the first way to get around the absence of a 
satisfying duality theory for the problem~\eqref{KantInf}. Indeed, as we said, the energy
$$ \gamma \mapsto \|y-x\|_{L^\infty_\gamma} $$
is non-linear, and not even convex, with respect to the transport plan $\gamma$, which makes difficult to guess a dual formulation 
of~\eqref{KantInf}, and this is why this alternative method using cyclical monotonicity together with a regularity property of the transport plan has been introduced. 
Later on, similar ``duality free'' methods allowed to prove existence of optimal transport maps for many non-standard cases 
of cost functions, including the Monge's distance cost for arbitrary norms (see for example~\cite{ChaDePJEMS, ChaDePDuke, JimSan, BerPuel}).

However, looking for a (sort of) dual formulation of~\eqref{KantInf} is natural in order to try to find characterization of optimal transport plans: indeed,
in the classical theory of optimal transportation, the duality method plays 
a crucial role not only in the proof of existence of an optimal transport map 
but also in the characterization of this map.
Heuristically, the reasoning which leads to 
the characterization is the following: from the duality formula
$$ \inf\left\{ \int c\dx\gamma \,:\, \gamma\in\Pi(\mu,\nu) \right\} = \sup\left\{ \int \phi\dx\mu + \int\psi\dx\nu \,:\, \phi(x)+\psi(y)\leq c(x,y)\right\}, $$
one can deduce that, if $\gamma$ and $(\phi,\psi)$ are respectively a minimizer and a maximizing pair for these problems,
 the condition $\phi\oplus\psi \leq c$ is saturated (in sense that equality holds)
 on the support of $\gamma$. Therefore, if $c$ and $\phi$ are regular enough, one can deduce
$$ \nabla_1 c(x,y) = \nabla\phi(x) \quad\text{for } \gamma\text{-a.e.}~(x,y)\in(\R^d)^2. $$
If the cost function $c$ is such that this relationship may be inverted, one concludes that an optimal transport maps exists and is characterized by $\phi$ (namely, it is exactly the map $x\mapsto (\nabla_1 c(x,\cdot))^{-1}(\nabla\phi(x))$). Following these arguments, Brenier~\cite{Bre1,Bre2} proved the existence of an optimal map for the quadratic cost and that this map is induced by a convex potential; it must then be a solution of the {\it Monge-Amp\`ere equation}, which serves as basis for a whole regularity theory of optimal transport maps, {\it cf.}~\cite[Chapter~4]{Vil}. Similar results have been generalized to a much larger class of cost functions, first for strictly convex costs with respect to the difference~\cite{GanMcC1995,GanMcC1996,Caf1996} and later on for those satisfying the so-called {\it twist condition}, see \cite{FatFig2010, ChaDeP2014}. \smallskip

These arguments cannot be reproduced for the cost functionals for which the existence of solutions $(\phi,\psi)$ is not known (see for instance~\cite{BerPraPue,LouPraZei} 
for recent progresses in this direction in special cases) and even, for our $L^\infty$ problem~\eqref{KantInf}, the meaning of ``dual formulation'' is itself not clear. Yet, although the functional 
$\gamma\mapsto \|y-x\|_{L^\infty_\gamma}$ is not convex, it is still level-convex 
(in the sense that the level sets of the functional are convex sets and this, sometimes, goes under the name of quasi-convex).
 Then it is still possible to consider a (sort of) duality theory. This theory has been recently introduced and investigated by Barron, Bocea and Jensen in \cite{BarBocJen2017} 
where it is proven that the minimal value of~\eqref{KantInf} is equal to
\begin{equation} \label{PbDualIntro} \inf\limits_{\lambda \geq 0} \left(\sup\left\{ \lambda + \int\phi\dx\mu + \int\psi\dx\nu \,:\, \begin{array}{cc} \phi(x)+\psi(y) \leq 0 \\ \text{whenever } |y-x|\leq \lambda \end{array} \right\}\right), \end{equation}
and that the infimum with respect to $\lambda$ in~\eqref{PbDualIntro} is attained for $\lambda = \W_\infty(\mu,\nu)$ (see the next section for more details). It follows that  $\W_\infty(\mu,\nu)$ is the smallest $\lambda$ such that 
$$\sup\left\{ \int\phi\dx\mu + \int\psi\dx\nu \,:\, \phi(x)+\psi(y) \leq 0  \text{ whenever } |y-x|\leq \lambda \right\}=0.$$
Unfortunately, the supremum with respect to $(\phi,\psi)$ always admits a trivial solution $\phi=\psi=0$. It is  
therefore important to look for the ``most interesting solutions'' of this dual problem, and to study 
what information can be obtained from them about the optimal plans for~\eqref{KantInf}. \medskip

The aim of this paper is to provide, in the one-dimensional case,  a solution of the dual problem~\eqref{PbDualIntro} which is non constant as much as possible. More precisely, if $(\phi,\psi)$ is a solution of~\eqref{PbDualIntro} for $\lambda = \W_\infty$ and $\gamma$ is a fixed optimal transport plan for~\eqref{KantInf}, one can at least formally notice that $\phi$ is locally constant around any point~$x$ which is sent by $\gamma$ at smaller distance that $\W_\infty$ (see below a more precise statement in Prop.~\ref{RkLocConst}). In the present paper, we construct a pair $(\phi,\psi)$ of solutions of~\eqref{PbDualIntro} (see Theorem~\ref{theophipsi}) from a fixed infinite cyclically monotone transport plan $\bgamma$, and such that $\phi$ is a BV function whose derivative is exactly supported on the set of points which are moved by $\bgamma$ at maximal distance: in this sense, our solution has a maximal set of variation. As an application, we will prove that there exists a set of points of 
the support of the source measure which are displaced at maximal distance by all the optimal transport plans for~\eqref{PbDualIntro}; in particular, this set is shared by all the infinitely 
cyclically monotone transport plans. Moreover, we prove that any optimal transport plan for \eqref{KantInf} must exactly coincide with the infinitely cyclically 
monotone one on the set of points which are displaced at maximal distance. 
The proofs of both these results are based on the properties of  the non trivial solution $(\phi,\psi)$ that we construct for  the dual problem.

The potentials we construct only carry informations on the points which are moved at maximal distance by all the optimal plans.
This fact is not surprising: indeed, as we said, the points which are moved at non-maximal distance by a transport plan are in some sense ``not relevant'', since the plan can be modified on a small region around such a point without loosing the optimality. It is then natural that the characterizations provided by this dual formulation only impacts the region 
where modifications could alter the optimality of a transport plans. In other words, the ``zone where $\phi$ carries some  information'' (i.e.  the zone where $\phi$ is not locally constant) is exactly the ``minimal set of maximal displacement of optimal plans'', so that our result is nearly optimal in that sense.

\subsection*{Acknowledgements.} The research of the first author is part of the project 2010-A2TFX2 {\it Calcolo delle Variazioni} funded by the Italian Ministry of Research, and has been partially financed by the {\it Fondi di ricerca di ateneo} of the University of Pisa. The second author acknowledges the support of PGMO project MACRO, funded by EDF and {\it Fondation Math\'ematique Jacques Hadamard}, and of the {\it Laboratoire Ypatia des Sciences Math\'ematiques} (LYSM). He also acknowledges the hospitality of Universities of Pisa and of Florence during several research visits where part of this work has been done.

\section{Notations and the dual problem}

In this section, we quickly collect all the notations and known facts of measure theory and 
optimal transportation that we will use throughout the paper.

Let $X$ and $Y$ be two Polish spaces, and $\mu$, $\nu$ be two positive measures on $X$, $Y$ whose total masses are finite and equal. We denote $\Pi(\mu,\nu)$ the set of {\it transport plans} from $\mu$ to $\nu$, that is, 
the set of positive measures on $X\times Y$ satisfying
$$ \text{for any Borel sets }A \subset X \text{ and } B \subset Y,\; \gamma(A\times Y) = 
\mu(A)\text{ and }\gamma(X\times B) = \nu(B); $$
recall that this constraint can be reformulated as
$$ \text{for any } (u,v) \in C_b(X) \times C_b(Y),\quad \begin{array}{rl} &  \iint u(x)\dx\gamma(x,y) = \int u\dx\mu \\[2mm] \text{and} &  \iint v(y)\dx\gamma(x,y) = \int v\dx\nu. \end{array} $$
In our settings, $X = Y = \R^d$ (and $d=1$ along almost the whole paper) and the ``primal'' problem that we consider is the minimization of the supremal functional \eqref{KantInf} above. We will denote by $\O_\infty(\mu,\nu)$ the set of its minimizers. The definition of an {\it infinitely cyclically monotone transport plan} has been recalled in the introduction, see Definition~\ref{DefInfCyclPlan}. We will use by simplicity the abbreviation ``\icm plan''; recall that, from~\cite[Theorems 3.2 and 4.4]{ChaDePJuu2008}, we know that at least one such  plan exists provided $\mu,\nu$ are both compactly supported in $\R^d$, and that these plans are exactly those which are {\it restrictable solutions} of~\eqref{KantInf}, in sense given by Definition~\ref{RestSol}. \smallskip

The following ``dual problem'' was introduced in Theorem 2.3 and Remark 2.4 of \cite{BarBocJen2017}.

\begin{thm}[Duality formula for the $L^\infty$-optimal transport problem] Let $\mu,\nu \in \mathcal{P}(\R^d)$ be compactly supported. For any $\lambda >0$, denote by $\mathcal{U}_\lambda$ the set of couples $(\phi,\psi) \in L^1_\mu \times L^1_\nu$ 
such that, for $\mu$-a.e.~$x$ and $\nu$-a.e.~$y$, the inequality $\phi(x)+\psi(y) \leq 0$ holds whenever $|y-x|\leq \lambda$. Then, 
$$  \W_\infty(\mu,\nu) = \inf\limits_{\lambda \geq 0} \left( \sup\left\{ \lambda + \int \phi\dx\mu+\int\psi\dx\nu \,:\, (\phi,\psi) \in \mathcal{U}_\lambda\right\}\right), $$
and the infimum with respect to $\lambda$ is attained for $\lambda = \W_\infty(\mu,\nu)$.
\end{thm}
For shorter notations we introduce $\lambda_C := \W_\infty(\mu,\nu)$ where $C$ stands for critical.
As a consequence
\begin{equation} \label{DualPb} \max\left\{ \int\phi\dx\mu + \int\psi\dx\nu \,:\, (\phi,\psi) \in \mathcal{U}_{\lambda_C} \right\}=0 \end{equation}
and we observe that, given an optimal transport plan $\gamma$ for the primal problem~\eqref{KantInf}, a pair $(\phi,\psi)$ is optimal for the dual problem~\eqref{DualPb} if and only if 
the equality $\phi(x)+\psi(y)=0$ holds for $\gamma$-a.e.~$(x,y)$.
Such functions $\phi,\psi$ will be then called {\it Kantorovich potentials.} We notice that among the maximizers for problem (\ref{DualPb}) above there are always $\phi\equiv 0$ and $\psi \equiv 0$.
The duality conditions impose a strong constraint on the variability of the Kantorovich potentials which make difficult to find non trivial couples. 
We explain this in the next Proposition.

\begin{prop} \label{RkLocConst}  Let  $(\phi,\psi) \in L^1_\mu \times L^1_\nu$ be Kantorovich potentials and let $\gamma$ be an optimal transport plan. 
Then $\mu$ is concentrated on a set $L$ such that if $\bx \in L$ has the property 
\begin{equation}\label{dp}
\max_{(\bx,y) \in \supp (\gamma)} |\bx-y| =\lambda<\lambda_C,
\end{equation}  
then there exists $\varepsilon>0$ such that $\varphi$ is constant in $B(\bx, \varepsilon)\cap L$
\end{prop}

\begin{proof} First we observe that since the distance is continuous and $\supp (\gamma)$ is closed, 
there exists $\varepsilon>0$ such that for every $x \in B(\bx, \varepsilon)$ 
$$\max_{(x,y) \in \supp (\gamma)} |x-y| <\lambda+ \frac{\lambda_C-\lambda}{2}<\lambda_C.$$
Denote by $D_c:=\{(x,y)\in \R^d \times \R^d \ | \ |x-y| \leq \lambda_C \},$ and let 
$\mathcal N_\mu$ and $\mathcal N_\nu$ be such that $\mu(\mathcal N_\mu)=\nu(\mathcal N_\nu)=0$ and 
$$ \varphi(x) +\psi(y) \leq 0 \text{ on } D_c \setminus (\mathcal N_\mu \times \R^d \cup \R^d \times \mathcal N_\nu).$$
Denote by $\Gamma$ the set on which $\gamma$ is concentrated and such that
$$ \varphi (x)+\psi(y) =0 \text{ for all } (x,y) \in \Gamma.$$

Finally let $\hat \Gamma:= \Gamma \setminus  (\mathcal N_\mu \times \R^d \cup \R^d \times \mathcal N_\nu).$
We will prove that $L= \pi_1 (\hat \Gamma)$ has the desired property.

We prove that if $\bx \in L$ is such that  $\max_{(\bx,y) \in \supp (\gamma)} |\bx-y| =\lambda<\lambda_C$ then for all 
$x \in B(\bx, \lambda_C-\lambda) \cap L$ 
$$ \varphi(x) \leq \varphi (\bx),$$
i.e. $\bx$ is a local maximum for $\varphi$ on the set $L$ which has full $\mu$-measure. 
Since $\bx \in L$ there exists $\by$ such that $(\bx,\by)\in \hat \Gamma$ (so  $\bx \not\in \mathcal N_\mu$ nor 
$\by \not\in \mathcal N _\nu$). 
Let $\tilde x \in L\cap B(\bx, \lambda_C-\lambda)$, we have $|\tilde x- \by| \leq \lambda_C$ and again $\tilde x \not\in \mathcal N_\mu$ and $\by \not\in \mathcal N _\nu$ it follows that 
$$ \varphi (\tilde x)+ \psi (\by) \leq 0$$ 
and then  
$$ \varphi(\tilde x) \leq -\psi(\by) = \varphi (\bx).$$
If, in the first step of the proof,  we choose $\varepsilon <\frac{\lambda_C-\lambda}{4}$ and we take $\tilde{x} \in B(\bx, \varepsilon) \cap L$, by symmetry, we also have 
$$\varphi (\bx) \leq \varphi(\tilde x)$$
which concludes the proof.
\end{proof}


\section{Construction of non-trivial potentials}


In this section, we construct, in the one-dimensional case, a couple of Kantorovich potentials $(\phi,\psi)$ which are not locally constant on 
the largest possible set. Our assumptions on the data are that $\mu,\nu$ are two probability measures on $\R$ with compact support and without atoms.

We will use the following notations:
\begin{itemize}
\item as before, we denote by $\lambda_C$ the optimal value of the problem~\eqref{KantInf}, and by $\Op_\infty$ the set of optimizers;
\item for any $\gamma\in\Op_\infty$, we introduce the sets
  $$ \M_\gamma^+ := \left\{ (x,y)\in\supp(\gamma) \,:\, y-x = \lambda_C  \right\}, $$
  $$ \M_\gamma^- := \left\{ (x,y)\in\supp(\gamma) \,:\, y-x = -\lambda_C  \right\},$$
  $$\M_\gamma := \M_\gamma^+\cup\M_\gamma^-  $$
$$ \text{and}\quad M_\gamma^+ = \pi_1(\M_\gamma^+),\; M_\gamma^- = \pi_1(\M_\gamma^-),\; M_\gamma = \pi_1(\M_\gamma), $$
where $\pi_1$ is the projection on the first variable ($\pi_1(x,y) = x$ for any $(x,y)\in\R^2$). We notice that $M_\gamma^+,\ M_\gamma^-$ are, respectively,  
the sets of points which are moved by $\gamma$ at maximal distance to the right or to the left. We will then call $M_\gamma$ the {\it Maximal displacement set of $\gamma$.} 
\end{itemize}
We also observe that $\M_\gamma^+$, $\M_\gamma^-$, $\M_\gamma$ and $M_\gamma^+$, $M_\gamma^-$, $M_\gamma$ are compact subsets of $\R^2$ and $\R$ respectively; moreover, the sets $M_\gamma^+$ and $M_\gamma^-$ 
are never simultaneously empty (and so is not $M_\gamma$).

\begin{rmk} \label{RmkIsolPts} Let $A \subset \R$ be a compact set a point $\bar{x} \in A$ will be called 
{\it right-extreme point of $A$} (resp.~{\it left-extreme point} of $A$) if there exists $\delta >0$ such that 
the interval $]\bar x,\bar x+\delta[$ (resp.~$]\bar x-\delta,\bar x[$) does not intersect $A$. We notice that since $A$ is closed, the set $\R\setminus A$ may be written as
$$ \R\setminus A = \bigcup\limits_i ]a_i,b_i[ $$
where the union is taken on an at most countable set of indexes~$i$. 
In particular, the left-extreme and right-extreme points are all part of the $a_i,b_i$, so that such points are at most countably many: 
since $\mu$ and $\nu$ have no atoms, the sets of left-extreme or right-extreme points of any closed set have always zero mass for $\mu,\nu$. \end{rmk}

\subsection{Properties of the maximal displacement set of \icm plans} 

In this paragraph, we fix an \icm plan $\bar\gamma$. Actually, in the one-dimensional context, it can be proven that such a~$\gamma$ is necessarily induced by a monotone map $T$, in sense that
$$ \gamma = (\text{id}\times T)_\#\mu \quad\text{where $T$ is non-decreasing and } T_\#\mu = \nu $$
and that such a map $T$ is unique up to $\mu$-negligible sets, see~\cite[Chapter 2]{SanOTAM}; the support of $\bgamma$ is then exactly the set of points
$$ \left\{ \left(x,\lim\limits_{\substack{t \to x \\ t<x}} T(t)\right) \,:\, x \in \supp\mu \right\} \bigcup \left\{ \left(x,\lim\limits_{\substack{t \to x \\ t>x}} T(t)\right) \,:\, x \in \supp\mu \right\}  $$
for a well-chosen representative of $T$. However, since the only property that we really use is the formula~\eqref{InfCycl} itself, we will not enter into more details 
concerning the characterization of  $\supp \bgamma$. 
 
We can then prove some additional properties of the maximal displacement set.
\begin{lem}\label{distance}   For any $\bar x \in M_\bgamma^+$, the intersection $]\bar x,\bar x+2\lambda_C[ \cap M_\bgamma^-$ is empty. Similarly, if  $\bar x \in M_\bgamma^-$, 
then $]\bar x-2 \lambda_C,\bar x[ \cap M_\bgamma^+ = \emptyset$.
\end{lem}  
\begin{proof}  We give a proof of the first case, being the second analogous. 
Assume by contradiction that there exists some $z \in M_\bgamma^-$ such that $\bar x < z <  \bar x+2\lambda_C$. Both $(z,z-\lambda_C)$ and $(\bar x,\bar x+\lambda_C)$ 
belong to the support of $\bgamma$, so that the \icm property implies
\begin{equation} \lambda_C = \max\Big(|(z-\lambda_C) - z|,|(\bar x+\lambda_C)-\bar x|\Big) \leq \max\Big(|(z-\lambda_C) - \bar x|,|(\bar x+\lambda_C)-z|\Big). \label{inf2cmxz} \end{equation}
On the other hand, the fact that $  \bar x < z < \bar x+2\lambda_C $ implies immediately that the right-handside in \eqref{inf2cmxz} is smaller than $\lambda_C$ and this is a contradiction.
 \end{proof} 

\begin{lem} \label{finite} The set $M_{\bar\gamma}^+\cap M_{\bar\gamma}^-$ is finite. \end{lem}
\begin{proof} Since $M_{\bar\gamma}^+\cap M_{\bar\gamma}^-$ is included in the support of $\mu$, which is bounded, it is sufficient to prove that all of its points are isolated, and this is true since, by Lemma \ref{distance}, above, if  $\bar{x} \in  M_\gamma^+\cap M_\gamma^-$, then  $(M_\gamma^+\cap M_\gamma^- ) \cap ]\bar{x},\bar x+2\lambda_C[= \emptyset $ and $(M_\gamma^+\cap M_\gamma^- ) \cap ]\bar x-2 \lambda_C,\bar x[=\emptyset$.
\end{proof}

\subsection{Construction of a non-trivial potential}

In this paragraph, starting from an \icm optimal transport plan $\bar\gamma$, 
we introduce a couple $(\phi,\psi)$ of non-trivial solutions of the dual problem \eqref{DualPb}.
The couple enjoys the property that the points around which $\phi$ is not locally constant 
are exactly those that $\bar\gamma$ moves at maximal distance i.e. those of $M_\bgamma$. We will need the following lemma of measure theory.

\begin{lem} \label{LemRhoPM} There exists two positive measures $\rho^+$, $\rho^-$ on $\R$, having finite mass and such that, denoting by $\rho = \rho^+-\rho^-$, the following properties are satisfied:
\begin{enumerate}[label={\normalfont (\roman*)}]
\item $\supp \rho^+=M_\bgamma^+$, $\supp \rho^-=M_\bgamma^-$ and $\supp \rho=M_\bgamma$;
\item for any point $\bx$ of $M_\bgamma^+$ which is a left-extreme or a right-extreme of $M_\bgamma^+$, we have $\rho(\{\bar x\}) >0$.
\end{enumerate} \end{lem}

\begin{proof} By lemma~\ref{finite}, the set $M_\bgamma^+ \cap M_\bgamma^-$ is finite, we denote by $\{z_1,\dots,z_N\}$ the (possibly empty) set of its elements. We also select two at most countable (and also possibly empty) families $(x_i)_i$, $(y_j)_j$ of points of $M_\bgamma^+\setminus M_\bgamma^-$, and of $M_\bgamma^-\setminus M_\bgamma^+$ which are dense in $M_\bgamma^+\setminus M_\bgamma^-$ and $M_\bgamma^-\setminus M_\bgamma^+$ respectively. Moreover, by Remark~\ref{RmkIsolPts}, the left-isolated (resp.~right-isolated) points of $M_\bgamma^+$ are at most countably many, so we can assume that all of them are part of the family $(x_i)_i$ or $(z_k)_k$. It is then straightforward to check that the measures
$$ \rho^+ := \sum\limits_i 2^{-i}\delta_{x_i} + \sum\limits_{k=1}^N \delta_{z_k} \quad \text{and} \quad \rho^- := \sum\limits_j 2^{-j}\delta_{y_j} + \sum\limits_{k=1}^N 2 \delta_{z_k}  $$
satisfy the required property. \end{proof}

Define, now, a candidate pair of Kantorovich potentials.

\begin{defi} For $x,y\in\R$, we define
$$  \phi_r(x) := -\rho((-\infty,x]) \;\text{and}\; \psi_r(y) := \inf\big\{ -\phi_r(x) : x\in[y-\lambda_C,y+\lambda_C]\big\}. $$
\end{defi}

\begin{thm}\label{theophipsi} The functions $\phi_r$ and $\psi_r$ defined above  satisfy the following properties:
\begin{enumerate}[label={\normalfont (\roman*)}]
\item $\phi_r$ has bounded variation (in particular, it has a left-limit and a right-limit at any point) and is right-continuous;
\item \label{PropDeriv} the support of $(\phi_r')^-$ is exactly $M_\bgamma^+$, and the support of $\phi_r'$ is exactly $M_\bgamma$;
\item \label{PropJump} for any point $\bar x \in M_\bgamma^+$ which is left-isolated or right-isolated in $M_\bgamma^+$, we have
$$ \lim\limits_{\substack{x\to\bar x \\ x<\bar x}} \phi(x) > \phi_r(\bar x) ;$$
\item the couple $(\phi_r,\psi_r)$ is a Kantorovich potential, {\it i.e.}~a maximizer of~\eqref{DualPb}.
\end{enumerate}
\end{thm}

In some sense, Property~\ref{PropDeriv} of Theorem 3.6 means that the potential $\phi_r$ is ``optimal with respect to $\bgamma$'': indeed, Prop.~\ref{RkLocConst} suggests that, for 
any potential $\phi$ having bounded variation, the support of $\phi'$ should not be larger than $M_\gamma$, as $\phi$ should be locally constant anywhere else. Moreover, we will 
that the smallest possible $M_\gamma$ among the optimal plans $\gamma$ is actually achieved by the \icm transport plans ({\it cf.}~Th.~\ref{ThmMgamma} below).

\begin{proof}[Proof of Theorem~\ref{theophipsi}.] By definition $\phi_r$ is the cumulative distribution function of the measure $-\rho$, then it is a BV function 
whose distributional derivative is the measure $-\rho$; this and the definition of $\rho$ classically imply
the properties~(i) and~(ii). As for~(iii), let $\bar x$ be an isolated point of $M_\bgamma^+$. Then, for any $x<\bar x$,
$$ \phi_r(\bar x)-\phi_r(x) = -\rho(]x,\bar x]) = -\rho(\{\bar x\}) - \rho(]x,\bar x [). $$
By Property~(ii) of Lemma~\ref{LemRhoPM} we have $\rho(\{\bar x\})>0$ while the other term in the last equality vanishes as $x\to \bar x$, which proves~\ref{PropJump}.

It remains to show that $(\phi_r,\psi_r)$ is a pair of Kantorovich potentials. First of all, from the definition of $\psi_r$, it follows that
$$ \text{for any }(x,y)\in\R^2\text{ with } |y-x|\leq\lambda_C, \quad \phi_r(x) + \psi_r(y) \leq 0.  $$
This proves that $(\phi_r,\psi_r)$ is an admissible couple for the dual problem. We moreover claim that
\begin{equation} \phi_r(\bar x) + \psi_r(\bar y) = 0 \;  \text{for } \bgamma\text{-a.e.}~(x,y); \label{egphipsi} \end{equation}
more precisely, we will prove that, if $(\bar x,\bar y) \in \supp\mu\times\supp\nu$ is such that~\eqref{egphipsi} does not hold, then either $\bar{x}$ belongs to an at most countable subset of $\supp\mu$ or $\bar{y}$ belongs to an at most countable subset of $\supp\nu$. This will be enough to prove~(iv) since it implies that
$$ \int \phi_r\dx\mu + \int \psi_r\dx\nu = \int (\phi_r(x)+\psi_r(y))\dx\bgamma(x,y) = 0,$$
so that $(\phi_r,\psi_r)$ is optimal.

Let $(\bar x,\bar y)$ be a point of $\supp(\bgamma)$ such that $\phi_r(\bar x)+\psi_r(\bar y) < 0$. Then:
$$ -\phi_r(\bar x) > \psi_r(\bar y) = \inf\big\{ -\phi_r(x)  \,:\, x\in[\bar y-\lambda_C,\bar y+\lambda_C ]\big\}.  $$
We deduce that there exists $z\in [\bar y-\lambda_C,\bar y+\lambda_C]$ such that $\phi_r(z) > \phi_r(\bar x)$, and we distinguish the cases $z<\bar{x}$ and $z>\bar{x}$.

{\it First case: $z<\bar{x}$.}  In this case, 
the definition of~$z$ reads
$$ \phi_r(\bar x)-\phi_r(z) = -\rho(]z,\bar x]) <0.  $$
In particular, there exists $z' \in ]z,\bar x]$ which belongs to $M_\bgamma^+$, so 
that $(z',z'+\lambda_C) \in \supp(\bgamma)$.  Then, notice that
\begin{equation}\label{firstmax} 
\max\big\{ |\bar y - \bar x|,|(z'+\lambda_C)-z'|\} = \lambda_C 
\end{equation}
and consider  
\begin{equation}\label{secmax}  
\max\big\{|\bar y-z'|,\  |(z'+\lambda_C)-\bar x|\}.
\end{equation}
By the cyclical monotonicity of $\bar\gamma$, the $\max$ in (\ref{secmax}) must be at least $\lambda_C$ too, meaning that either $|\bar y-z'|$ or $|(z'+\lambda_C)-\bar x|$ is larger or equal to $\lambda_C$. Now:
\begin{itemize}
\item From the fact that $z' \in (z,\bar x]$ and that $z,\bar x \in [\bar y-\lambda_C,\bar y+\lambda_C]$, we have
$$ -\lambda_C <z'-\bar y \leq \bar x-\bar y\leq \lambda_C. $$
Therefore, the inequality  $|\bar y-z'|\geq\lambda_C$ is only possible if $z'=\bar x = \bar y+\lambda_C$. This implies that $\bar x$ belongs to $M_\bgamma^+\cap M_\bgamma^-$, which is finite by Lemma~\ref{finite}.
\item On the other hand, using again that $z'\in ]z,\bar x]$ and $z,\bar x \in [\bar y-\lambda_C,\bar y+\lambda_C]$, we observe
$$ - \lambda_C \leq \bar y-\bar x < z'+\lambda_C-\bar x \leq \lambda_C. $$
The inequality $|(z'+\lambda_C)-\bar x| \geq \lambda_C$ can then only hold if $\bar x=z'$. In this case, both $(\bar x, \bar y)$ and $(\bar x, \bar x+ \lambda _C)$ belong to the support of $\bgamma$. As a consequence of \icm condition, for every $(x,y') \in \supp \bgamma$ with $\bar y < y' < \bar x + \lambda_C$ it must hold $x=\bar x$ (in fact, if $x< \bar x $ the couples $(x,y')$ and $(\bar x, \bar y)$ would violate the condition while if 
$\bx < x$ the bad couples would be $(x,y')$ and $(\bar x, \bar x + \lambda _C)$). It follows that 
\begin{align*} 0=\mu(\{\bar x\}) & = \bar\gamma( \{\bar x\}\times \R) \\
  &\geq \bar\gamma( \{\bar x\}\times ]\bar y,\bar x+\lambda_C[) \\
  &= \bar\gamma( \R \times ]\bar y,\bar x+\lambda_C[ ) \\
  &= \nu(]\bar y,\bar x+\lambda_C[). \end{align*}
Then $\bar y$ is a right-isolated point of $\supp\nu$ which, as discussed in Remark \ref{RmkIsolPts}, are at most countably many.
\end{itemize}

{\it Second case: $z>\bar{x}$.} This case is treated in a very similar way as the previous one. First of all, we start by deducing from the inequality $\phi_r(z)>\phi_r(\bar x)$ that there exists $z'\in(\bar x,z]$ such that $(z',z'-\lambda_C) \in\supp(\bgamma)$. Then, by the \icm condition,
\begin{equation} \lambda_C = \max\big\{ |\bar y-\bar x|,|(z'-\lambda_C)-z'|\big\} \leq \max\big\{|\bar y -z'|,|(z'-\lambda_C)-\bar x|\}, \label{inegzprime} \end{equation}
so that either $|\bar y-z'|$ or $|z'-\lambda_C-\bar x|$ is at least equal to $\lambda_C$. On the other hand, the fact that $z'\in(\bar x,z]$ with $\bar y-\lambda_C \leq \bar x \leq z \leq \bar y+\lambda_C$ enforces
\begin{align} & -\lambda_C < z'-\lambda_C-\bar x \leq \bar y-\bar x \leq \lambda_C.  \label{Case2Ineg1} \\
  \text{and} & -\lambda_C < z'-\bar y \leq \lambda_C  \label{Case2Ineg2} \end{align}
Therefore, the condition \eqref{inegzprime} is only satisfied in the two following cases:
\begin{itemize}
\item there is equality in the two last inequalities of~\eqref{Case2Ineg1}: in this case, we have simultaneously $z' = \bar y+\lambda_C$ and $\bar x = \bar y-\lambda_C$. Therefore, $(\bar y-\lambda_C,\bar y)$ and $(\bar y+\lambda_C,\bar y)$ both belong to the support of $\bgamma$: this means that $\bar y \in M_{\tilde \gamma}^+\cap M_{\tilde\gamma}^-$, where $\tilde\gamma$ is the ``symmetric plan'' of $\bgamma$, that is
$$ \tilde\gamma = \tau_\#\bgamma \quad \text{with } \tau(x,y) = (y,x).  $$
The transport plan $\tilde\gamma$ (which belongs to $\Pi(\nu,\mu)$) being itself \icm, Lemma~\ref{finite} also applies, yielding finiteness of $M_{\tilde\gamma}^+\cap M_{\tilde\gamma}^-$. Therefore, the equality case in~\eqref{Case2Ineg1} is only possible for a finite number of $\bar y$.
\item If there is equality in the last inequality of~\eqref{Case2Ineg2}, we deduce that both $(\bar x,\bar y)$ and $(\bar y+\lambda_C, \bar y)$ belong to the support of $\bgamma$, with $\bar x < \bar y+\lambda_C$.  As in the previous case, we conclude that $\bar x$ is a boundary point of one of the connected components of $\R\setminus (\supp\mu)$, which are at most countably many thanks to Remark~\ref{RmkIsolPts}, concluding the proof of~(iv).\qedhere
\end{itemize}
\end{proof}

\section{Application: existence of a minimal set of 
maximal displacements} \label{SectAppl}

In this section, we will use the pair of Kantorovich potentials provided by Theorem~\ref{theophipsi} to get some informations on general 
optimal transport plans for~\eqref{KantInf} (not necessarily \icm~). Since any optimal plan $\gamma$ is concentrated on the set
$$ \Big\{ (x,y) \in \R^2 \,:\, \phi_r(x)+\psi_r(y) = 0 \text{ and } |y-x|\leq \lambda_C\Big\}, $$
it is natural to start by collecting some properties of this set. This is the goal of the following proposition: 
roughly speaking, it suggests that, if $\bar x$ belongs to $M_\bgamma^+$, its only possible ``image'' by an optimal transport plan $\gamma$ is $\bar x+\lambda_C$.

\begin{prop} \label{PropBr} For any $x$, denote by
$$ B_r(x) := \Big\{ y\in [x-\lambda_C,x+\lambda_C] \,:\, \phi_r(x)+\psi_r(y) = 0 \Big\}. $$
Let $\bar x \in M_\bgamma^+$. Then, for any $x$ such that $\bar x\leq x <\bar x+2\lambda_C$, we have $B_r(x) \subseteq [\bar x+\lambda_C,x+\lambda_C]$. In particular, the set $B_r(\bar x)$ is reduced to $\{\bar x+\lambda_C\}$.
\end{prop}

\begin{proof} Let us fix $\bar x\in M_\bgamma^+$ and select $x \in [\bar x,\bar x+2\lambda_C[$. By definition, all the elements of $B_r(x)$ are at most equal to $x+\lambda_C$, thus we only have to prove that $B_r(x)$ does not contain any $y$ which is smaller than $\bar x+\lambda_C$. Consider then $y$ with $x-\lambda_C \leq y <\bar x+\lambda_C$, and let us prove that the equality $\phi_r(x)+\psi_r(y)=0$ can not hold. We observe that, by Lemma \ref{distance}, if $x>\bar x$ then $\rho^-(]\bar x,x]) = 0$, so that
\begin{equation} \phi_r(\bar x) \geq \phi_r(x)  \label{phirxbarre} \end{equation}
which is of course also true when $x = \bar x$. On the other hand, by definition of $\psi_r$ and since $[y-\lambda_C,\bar x[\subset[y-\lambda_C,y+\lambda_C]$, we also have 
\begin{equation} \psi_r(y) 
 \leq \inf\Big\{ -\phi_r(z) \,:\, y-\lambda_C \leq z < \bar x\Big\}. \label{inegpsir} \end{equation}

From~\eqref{phirxbarre} and~\eqref{inegpsir}, it is clear that proving that
\begin{equation} \exists z\in[y-\lambda_C,\bar x[,\, \phi_r(z) > \phi_r(\bar x) \label{existsz} \end{equation}
would be enough to conclude that $\psi_r(y)<-\phi_r(x)$. 
We prove \eqref{existsz} by separating two cases.

\begin{itemize}
\item {\it \underline{First case:} $\bar x$ is a left-isolated point of $M_\bgamma^+$}. In this case, Property~\ref{PropJump} of Theorem \ref{theophipsi} ensures
$$ \lim\limits_{\substack{z \to \bar x\\ z<\bar x}} \phi_r(z) > \phi_r(\bar x) $$
which is clearly enough to find the desired~$z$.

\item {\it \underline{Second case:} $\bar x$ is not a left-isolated point of $M_\bgamma^+$.} Let us select 
some $\tilde x \in ]\bar x-\lambda_C,\bar x[ \cap M_\bgamma^+$.  By applying Lemma \ref{distance} to $\tilde x$, we know that $]\tilde x,\bar x]\cap M_\bgamma^-=\emptyset$. From this fact and since 
$\bar x$ is a left-cluster point of $M_\bgamma^+$, we deduce
$$ \rho^+(]\tilde x,\bar x]) > 0 \quad \text{and} \quad \rho^-(]\tilde x,\bar x]) = 0 $$
which implies $\phi_r(\tilde x)>\phi_r(\bar x)$, as expected. \qedhere
\end{itemize}
\end{proof}

We are now able to state the next main result of this paper, which in particular asserts the existence of a ``minimal set of maximal displacements'', which is shared by all the optimal transport plans:

\begin{thm} \label{ThmMgamma} Let $\gamma, \bgamma \in \Op_\infty$ be optimal transport plans and assume that 
$\bgamma$ satisfies the \icm-condition. 
Then
$$M_\bgamma^+ \subseteq M_\gamma^+ \quad \text{and} \quad M_\bgamma^-\subseteq M_\gamma^-.$$
In other words, there are two compact sets $M^+$ and $M^-$ which are "minimal sets of maximal displacement for optimal plans", in sense that, for any optimal plan $\gamma$,
$$ M^+ \subseteq M_\bgamma^+ \quad\text{and}\quad M^-\subseteq M_\bgamma^- $$
with equality if $\gamma$ is an \icm plan (and in particular, all the \icm plans have same set of maximal displacements).
\end{thm}

\begin{proof} First we consider some consequences of~Prop.~\ref{PropBr} for $\gamma$. \smallskip

{\it \underline{Step I.} For any $\bar x \in M_\bgamma^+$, the quarter of plane $Q := [\bar x,+\infty) \times (-\infty,\bar x+\lambda_C]$ has zero mass for $\gamma$.} To prove this result, we first observe that, by the optimality of $\gamma$ and of $(\phi_r,\psi_r)$ it follows that $\gamma$ is concentrated on the~set
$$ B := \Big\{ (x,y) \in \R^2 \,:\, |y-x|\leq \lambda_C \text{ and } \phi_r(x)+\psi_r(y) = 0 \Big\} $$
which can be written as $\Big\{ (x,y) \,:\, y \in B_r(x)\Big\}$. Now, for $x \geq \bar x$ and $y \in B_r(x)$ we have:
\begin{itemize}
\item if $x < \bar x+2\lambda_C$, then Prop.~\ref{PropBr} ensures that $y \geq \bar x+\lambda_C$;
\item if $x \geq x+2\lambda_C$, the same holds since $|y-x|\leq \lambda_C$.
\end{itemize}

This exactly means that $B$ has no intersection with the interior of the quarter of plane~$Q$, which has therefore zero mass for $\gamma$; since moreover the lines $\{\bar x\}\times \R$ and $\R \times \{\bar x+\lambda_C\}$ have also zero mass for $\gamma$ (because $\mu$ and $\nu$ are non atomic), the Step~I is proven. \smallskip

{\it \underline{Step II:} if there exists $\bar x \in M_\bgamma^+\setminus M_\gamma^+$, then such an~$\bar x$ satisfies $\mu([\bar x,\bar x+\delta]) = \nu([\bar x+\lambda_C-\delta,\bar x+ \lambda_C]) = 0$ for $\delta >0$ small enough.} Let $\bar x$ be such a point; since $\bar x \notin M_\gamma^+$, we can select $\delta >0$ such that the square with side-length $2\delta$ and centered in the point $(\bar x,\bar x+\lambda_C)$ has zero $\gamma$-measure. Without loss of generality we may also assume that $0<\delta<\lambda_C$. Since $\mu$ is the first marginal of $\gamma$, we then have
$$ \mu([\bar x,\bar x+\delta]) = \gamma([\bar x,\bar x+\delta]\times \R) .$$
Since $\bar x \in M_\bgamma^+$ the result of Step~I applies and gives $\gamma([\bar x,\bar x+\delta] \times (-\infty,\bar x+\lambda_C])=0$, so that
$$ \gamma([\bar x,\bar x+\delta]\times \R) = \gamma([\bar x,\bar x+\delta] \times [\bar x+\lambda_C,+\infty)). $$
Moreover, since $\gamma$ is an optimal transport plan, we have $|y-x|\leq \lambda_C$ for $\gamma$-a.e.~$(x,y)$; in particular, for $\gamma$-a.e.~$(x,y)$ with $x \leq \bar x+\delta$, we have $y \leq \bar x+\delta+\lambda_C$. Consequently,
$$ \gamma([\bar x,\bar x+\delta] \times [\bar x+\lambda_C,+\infty)) = \gamma([\bar x,\bar x+\delta] \times [\bar x+\lambda_C,\bar x +\lambda_C+\delta]), $$
which is zero since,  by assumption on $\bar x$ and $\delta$, the square $[\bar x-\delta, \bx+ \delta]\times [\bar x+\lambda_C -\delta, \bx+\lambda_C+\delta]$ has zero mass for $\gamma$. The three last equalities imply that $\mu([\bar x,\bar x+\delta]) = 0$. The proof of the other equality is similar, and may be described as follows:
$$ \begin{array}{ll}
\nu([\bar x+\lambda_C-\delta, \bar x+\lambda_C]) \! & = \gamma(\R\times [\bar x+\lambda_C-\delta,\bar x+\lambda_C]) \; \text{since } \nu = (\pi_2)_\#\gamma \\[2mm]
 & = \gamma((-\infty,\bar x] \times [\bar x+\lambda_C-\delta,\bar x+\lambda_C] ) \; \text{by Step~I} \\[2mm]
 & = \gamma([\bar x-\delta,\bar x]\times [\bar x+\lambda_C-\delta,\bar x+\lambda_C] ) \; \text{since } |y-x|\leq\lambda_C \; \gamma\text{-a.e.} \\[2mm]
 & = 0 \; \text{by assumption.} \end{array} \smallskip $$

{\it \underline{Step III:}  $M_\bgamma^+\subset M_\gamma^+$.} By contradiction, assume that there exists $\bar x\in M_\bgamma^+$ which does not belong to $M_\gamma^+$. Since $\bgamma$ is concentrated on the half-space $\Big\{ (x,y) \in \R^2 \,:\, y \leq x+\lambda_C\Big\}$, we notice that
$$ \bgamma([\bar x-\delta,\bar x]\times [\bar x+\lambda_C,\bar x+\lambda_C+\delta]) = 0 $$
because, in this square, only the point $(\bar x,\bar x+\lambda_C)$  satisfies $|y-x|\leq \lambda_C$ and $\bgamma$ has no atom (because $\mu$ and $\nu$ don't).

From this property, and from the fact that $\bgamma$ gives no mass to any horizontal or vertical line, it follows
\begin{equation} \begin{array}{ll} \bgamma ([\bx- \delta, \bx+\delta]\times [\bar x+\lambda_C - \delta, \bx+\lambda_C+\delta])  = & \bgamma([\bar x,\bar x+\delta]\times 
[\bx+\lambda_C - \delta,\bx+\lambda_C + \delta ]) \\[2mm]
 &  + \, \bgamma([\bx-\delta,\bx]\times [\bx+\lambda_C-\delta,\bar x+\lambda_C]).
\end{array} \label{bgammarectangles} \end{equation}
But, from the result of Step~II, it follows that
$$ \begin{array}{ll} \bgamma([\bar x,\bar x+\delta]\times [\bar x+\lambda_C \pm \delta]) & \leq \bgamma([\bar x,\bar x+\delta]\times \R) \\[2mm]
 & = \mu([\bar x,\bar x+\delta]) \\[2mm]
 & = 0 \end{array} $$

and that, similarly, $\bgamma([\bar x-\delta,\bar x]\times [\bar x-\lambda_C-\delta,\bar x-\lambda_C])$ is zero. The equality~\eqref{bgammarectangles} therefore implies that the open square with side-length~$2\delta$ and with center~$(\bar x,\bar x+\lambda_C)$ has zero mass for~$\bgamma$, which gives the desired contradiction and concludes Step~III. \smallskip

{\it \underline{Step IV:} $M_\bgamma^- \subset M_\gamma^-$.} Consider the ``opposite transport plans''
$$ \left\{ \begin{array}{l} \bgamma_\neg := \chi_\#\bgamma  \\  \gamma_\neg := \chi_\#\gamma  \end{array}\right. \qquad \text{where} \; \chi(x,y) = (-x,-y).  $$
It is clear that $\bgamma_\neg,\gamma_\neg$ have same marginals which are atomless and compactly supported (they are $\dx\mu(-x)$ and $\dx\nu(-y)$), have same maximal displacement which is still equal to $\lambda_C$ and that $\bgamma_\neg$ is still an \icm plan, and that
$$ x \in M_{\bgamma_\neg}^+ \Longleftrightarrow (-x) \in M_\bgamma^- \quad\text{and}\quad x \in M_{\gamma_\neg}^+ \Longleftrightarrow (-x)\in M_\gamma^-.  $$
We can then apply the result if Steps I-III to the plans $\bgamma_\neg$ and $\gamma_\neg$, which concludes the proof.
\end{proof}

We conclude this paper with a description of the general structure of an optimal transport plans for~\eqref{KantInf}. This result is very similar to the one obtained in~\cite[Prop.~3.1]{DiMLou} in the $L^1$ context. The optimal plans are constrained to coincide with the graph of the translation of $\pm \lambda_C$ on the critical regions $M_\bgamma^\pm$ while on any maximal interval 
$I \subset \R\setminus M_\bgamma$, they are free to ``fill'' the whole region $\{|y-x|\leq \lambda_C\} \cap (I\times I\pm\lambda_C)$ provided the marginal constraint is satisfied. \smallskip

In order to make precise this description, let us introduce some additional notations. Let $a<b$ be the lower and upper bounds of $\supp\mu$ and $c<d$ be those of $\supp\nu$. We can then write
$$ ]a,b[\setminus M_\bgamma = \bigcup\limits_{i\in I} ]a_i,b_i[ $$
where $I$ is an at most countable set and, for each $i$, $]a_i,b_i[$ is a maximal open interval fully included in $]a,b[\setminus M_\bgamma$. In particular, all the points 
$a_i,b_i$ belong to $M_\bgamma \cup \{a,b\}$.

We then set, for each $i$,
$$ c_i = \left\{ \begin{array}{ll} c & \text{if } a_i = a, \\ a_i + \lambda_C & \text{if } a_i \in M_\bgamma^+ \text{ and } a_i\neq a, \\ a_i-\lambda_C & \text{ if } a_i \in M_\bgamma^-,\, a_i \notin M_\bgamma^+ \text{ and } a_i \neq a \end{array} \right.  $$
and similarly
$$ d_i = \left\{ \begin{array}{ll} d & \text{if } b_i = b, \\ b_i - \lambda_C & \text{if } b_i \in M_\bgamma^- \text{ and } b_i\neq b, \\ b_i+\lambda_C & \text{ if } b_i \in M_\bgamma^+,\, b_i \notin M_\bgamma^- \text{ and } b_i \neq b. \end{array} \right.  $$
Roughly speaking, each $c_i$ (resp.~$d_i$) is ``the image'' of $a_i$ (resp.~$b_i$) by the infinitely cyclically monotone transport plan $\bgamma$. In particular, they satisfy the following property.

\begin{lem} \label{LemDisjoint} The intervals $]c_i,d_i[$, $i\in I$, are mutually disjoint. Moreover, for any $i$, $M_\bgamma^+ + \lambda_C$ and 
$M_\bgamma^--\lambda_C$ have also empty intersection with $]c_i,d_i[$. \end{lem}

\begin{proof} Fix $i \neq j \in I$; assume for instance that $a_i<b_i \leq a_j<b_j$. It is then enough to prove that $d_i \leq c_j$. From the definition of $d_i,c_j$, it turns out that the only case we need to treat is
$$ d_i  = b_i+\lambda_C \quad\text{and}\quad c_j = a_j-\lambda_C.  $$
This case corresponds to to
$$ a_j \neq a, \, a_j \in M_\bgamma^-\setminus M_\bgamma^+,\, b_i \neq b,\, \in M_\bgamma^+\setminus M_\bgamma^- $$
and in particular implies that $b_i < a_j$ (otherwise these equal points would belong to $M_\bgamma^+\cap M_\gamma^-$). From the fact that the points $(b_i,b_i+\lambda_C)$ and $(a_j,a_j-\lambda_C)$ both belong to the support of the \icm plan $\bgamma$, we then deduce
\begin{equation} |b_i+\lambda_C-a_j| \geq \lambda_C. \label{ICMajbi} \end{equation}
Assume now that $d_i > c_j$, {\it i.e.}~that $b_i+\lambda_C > a_j-\lambda_C$. 
Then, since moreover $b_i-a_j >0$, we deduce that the inequality~\eqref{ICMajbi} does not hold, a contradiction. \smallskip

As for the second point, let $i \in I$, consider $\bar y \in (c_i,d_i)$ and let us prove that $\bar x := \bar y-\lambda_C$ does not belong to $M_\bgamma^+$. By definition of $(a_i,b_i)$, this is clearly true if $a_i < \bar x < b_i$, thus we only need to consider the cases $\bar x \leq a_i$ and $\bar x \geq b_i$.

Assume first that $\bar x \leq a_i$. If this is an equality, this implies $a_i \in M_\bgamma^+$ thus $c_i = a_i+\lambda_C$ (this is also true in this case $a_i = a \in M_\bgamma^+)$: since $\bar y = \bar x+\lambda_C = a_i+\lambda_C > c_i$, this is impossible. If now the inequality $\bar x\leq a_i$ is strict, then one has
$$ \bar x = \bar y-\lambda_C < a_i \quad \text{and}\quad c_i < \bar y $$
from which we deduce $c_i-\lambda_C < a_i$. In particular, we are in the cases ``$a_i = a$'' or ``$a_i \in M_\bgamma^-\setminus M_\bgamma^+$'' of the definition of $c_i$. If $a_i = a$, the inequality $\bar x<a_i$ leads to $\bar x \notin \supp\mu$ so that $\bar x \notin M_\bgamma^+$, as expected; if now $a_i \in M_\bgamma^-\setminus M_\bgamma^+$, then from Lemma~\ref{distance}, we deduce
$$ (a_i-2\lambda_C,a_i)\cap M_\bgamma^+ = \emptyset. $$
Since $\bar x < a_i$, $c_i < \bar y$, $c_i = a_i-\lambda_C$ and $\bar y = \bar x+\lambda_C$, it is easy to see that $\bar x \in (a_i-2\lambda_C,a_i)$, so that it cannot belong to $M_\bgamma^+$.

In the case $\bar x \geq b_i$ we have
$$ \bar x \geq b_i,\quad \bar x = \bar y-\lambda_C \quad\text{and}\quad \bar y < d_i $$
from which we deduce $d_i>b_i+\lambda_C$. But from the definition of $d_i$, we can see that $(b_i,d_i)$ must belong to the support of $\bgamma$ so that $d_i\leq b_i+\lambda_C$, which is a contradiction. \smallskip

The last part of Lemma~\ref{LemDisjoint}, namely the fact that $\bar y+\lambda_C$ never belongs to $M_\bgamma^-$ for any $\bar y \in (c_i,d_i)$, can be proven in a very similar way. \end{proof}

Having at hand the last notations and the result of Lemma~\ref{LemDisjoint}, we can now state the structural result which characterizes all the one-dimensional optimal plans for~\eqref{KantInf}.

\begin{thm} For each $i$, denote by
$$ \mu_i := \mu\res[a_i,b_i] \quad\text{and}\quad \nu_i = \nu\res[c_i,d_i]. $$
Then, for each $i$, $\mu_i$ and $\nu_i$ have same total mass. Moreover, the optimal plans $\gamma$ for the problem~\eqref{KantInf} are exactly those which can be written as
\begin{equation} \label{EqnStructOptPlans} \gamma = (\id\times(\id+\lambda_C))_\#(\mu\res M_\bgamma^+) + (\id\times(\id-\lambda_C))_\#(\mu\res M_\bgamma^-) + \sum\limits_i \gamma_i \end{equation}
where, for each $i$, $\gamma_i$ is a transport plan from $\mu_i$ to $\nu_i$ such that $|y-x|\leq \lambda_C$ for $\gamma_i$-a.e.~$(x,y)$.
\end{thm}

\begin{proof} 

Let $\gamma \in \mathcal{O}_\infty(\mu,\nu)$, we will prove that it satisfies~\eqref{EqnStructOptPlans}. Again we divide the proof in several steps.

{\it \underline{Step I:} the measure $\gamma\res(M_\bgamma^+\times\R)$ (resp.~$\gamma\res(M_\bgamma^-\times \R)$) is concentrated on the line $\{y=x+\lambda_C\}$ (resp.~$\{y=x-\lambda_C\}$).} Let $(\bar x,\bar y)$ be such that
$$ \text{for any }\eps>0,\; \gamma\Big(([\bar x-\eps, \bx + \eps]\cap M_\bgamma^+)\times [\bar y\pm\eps]\Big) >0, $$
we prove that, for $\gamma$-a.e.~such $(\bar x,\bar y)$, $\bar y$ is necessarily equal to $\bar x + \lambda_C$. Assume by contradiction that $\bar y <\bar x+\lambda_C$ (the converse inequality can not hold since $\gamma \in \mathcal{O}_\infty(\mu,\nu)$); without loss of generality, we may assume that $\bar x$ is not a left-isolated point of $M_\bgamma^+$, since such points are countably many. Select then $\tilde x \in M_\bgamma^+$ such that
$$ 0 < \bar x-\tilde x < (\bar x+\lambda_C)-\bar y. $$
Since $\tilde x \in M_\bgamma^+$ by Th.~\ref{ThmMgamma} implies that
$$ \gamma([\tilde x,+\infty)\times(-\infty,\tilde x+\lambda_C]) = 0. $$
In particular, for $\eps>0$ such that $\bar x-\eps> \tilde x$ and $\bar y < (\tilde x+\lambda_C)-\eps$, we obtain $\gamma([\bar x-\eps, \bx+\eps]\times [\bar y-\eps, \by+\eps]) = 0$, in contradiction withthe starting point. This proves that $\gamma\res (M_\bgamma^+\times \R)$ is concentrated on the line $\{ y = x+\lambda_C\}$; the same for $\gamma\res (M_\bgamma^-\times \R)$ can then be obtained by considering the ``opposite transport plans'' $\bgamma_{\text{neg}},\gamma_{\text{neg}}$, as in Step~IV of the proof of Theorem~\ref{ThmMgamma}.

{\it \underline{Step II:} for any $i$, the plan $\gamma\res([a_i,b_i]\times \R)$ is concentrated on $[a_i,b_i]\times[c_i,d_i]$.} Let $i \in I$, first we prove that
\begin{equation} \label{GammaAiCi} \gamma\res([a_i,+\infty)\times \R) \text{ is concentrated on } [a_i,+\infty)\times [c_i,+\infty). \end{equation}
By definition of $c_i$, this is obvious if $a_i = a$; moreover, in the case $c_i = a_i-\lambda_C$, we have for $\gamma$-a.e.~$(x,y)$ with $x \geq a_i$:
$$ y \geq x -\lambda_C \geq a_i-\lambda_C = c_i.  $$
Then we need to look at the case $c_i = a_i+\lambda_C$, which corresponds to $a_i \in M_\bgamma^+$, and this, again, follows from the first step of Th.~\ref{ThmMgamma}.

To conclude we have to prove that
\begin{equation} \label{GammaBiDi} \gamma\res((-\infty,b_i]\times \R) \text{ is concentrated on } (-\infty,b_i]\times (-\infty,d_i]. \end{equation}
This is again obvious in the cases $b_i = b$ and $d_i = b_i+\lambda_C$, we thus have only to consider the case $d_i = b_i-\lambda_C$, which enforces $b_i \in M_\bgamma^-$. Again, it suffices to consider the negative transport plans $\gamma_{\text{neg}},\bgamma_{\text{neg}}$: observing then that $ -b_i \in M_{\bgamma_{\text{neg}}}^+ $, we apply Prop.~\ref{PropBr} to $\gamma_\neg,\bgamma_\neg$ to deduce
$$ \gamma_\neg([-b_i,+\infty)\times(-\infty,-d_i]) = 0  $$
which proves~\eqref{GammaBiDi}. \smallskip

{\it \underline{Step III:} conclusion.} Let us define
$$ \gamma^- := \gamma\res(M_\bgamma^-\times\R),\; \gamma^+ := \gamma\res(M_\bgamma^+\times \R) \;\text{and}\; \gamma_i := \gamma\res([a_i,b_i]\times \R) \text{ for each }i.  $$
It is clear that $\gamma = \gamma^-+\gamma^++\sum\limits_i \gamma_i$, and that $\gamma^-,\gamma^+$ have $\mu\res M_\bgamma^-,\mu\res M_\bgamma^+$ as respective first marginal. Moreover, Step~I proves that $\gamma^-$ (resp.~$\gamma^+$) is the transport plan induced by $\id-\lambda_C$ (resp.~$\id+\lambda_C$). Therefore, it remains to prove that, for each $i$, $\gamma_i$ has $\mu_i,\nu_i$ as marginals. Let us then fix $i\in I$; again the fact that $(\pi_1)_\#\gamma_i = \mu_i$ for each $i$ is a direct consequence of the definition of $\gamma_i$. Moreover, $\gamma_i$ is dominated by $\gamma$ and, by Step~II, is concentrated on $(a_i,b_i)\times (c_i,d_i)$: therefore
$$ (\pi_2)_\#\gamma_i \leq \Big((\pi_2)_\#\gamma\Big)\res (c_i,d_i) = \nu_i. $$
Assume now that this last inequality is strict: in other words, there exists a Borel set $B\subset (c_i,d_i)$ such that
$$ \gamma_i(\R\times B) < \nu(B). $$
By Lemma~\ref{LemDisjoint}, $B$ does not meet any $(c_j,d_j)$ with $j\neq i$ and neither meets $M_\bgamma^++\lambda_C$, $M_\bgamma^--\lambda_C$. Therefore the result of Step~I also implies
$$ \gamma^-(\R\times B) = \gamma^+(\R\times B) = \gamma_j(\R\times B) = 0 \quad\text{for any } j\neq i. $$
Adding all of these equalities, we obtain $\gamma(\R\times B)<\nu(B)$, which is impossible since $\gamma \in \Pi(\mu,\nu)$. \smallskip

This proves that any optimal plan satisfies~\eqref{EqnStructOptPlans}. In particular, this implies that $\mu_i$ and $\nu_i$ have same total mass for any $i$, and that
$$ \mu\res M_\bgamma^++\mu\res M_\bgamma^- + \sum\limits_i \mu_i = \mu,\quad \nu\res(M_\bgamma^++\lambda_C)+\nu\res(M_\bgamma^--\lambda_C)+ \sum\limits_i \nu_i = \nu.$$
For any plan satisfying~\eqref{EqnStructOptPlans} and belonging to $\Pi(\mu,\nu)$, the confinement betwen the two lines $y=x+\lambda_C$ and $y=x-\lambda_C$ inplies the optimality, concluding the proof.
\end{proof}



\begin{thebibliography}{10}

\bibitem{AroCraJuu} G.~Aronsson, M.G.~Crandall, P.~Juutinen, 
A tour of the theory of absolute minimizing function, Bull. Amer. Math. Soc. 41-(4) (2004), 439--505

\bibitem{BarBocJen2017}
E.N.Barron, M.Bocea, R.R. Jensen, Duality for the $L^\infty$ optimal transport problem, 
Trans. Amer. Math. Soc., 369 (5) (2017), 3289--3323

\bibitem{BerPraPue} J.~Bertrand, A.~Pratelli, M.~Puel, Kantorovich potentials and continuity of total cost for relativistic cost functions,
preprint
  
\bibitem{BerPuel} J.~Bertrand, M.~Puel, Optimal transport with relativistic cost, 
Calc. Var. Partial Differential Equations,  46 (1-2) (2013), 353--374

\bibitem{Bre1} 
Y.~Brenier, D\'ecomposition polaire et r\'earrangement monotone des champs de
vecteurs (in French), C. R. Acad. Sci. Paris S\'er. I Math., 305~(19) (1987), 805--808

\bibitem{Bre2} Y.~Brenier,
Polar factorization and monotone rearrangement of vector-valued functions, 
Comm. Pure Appl. Math. 44~(4) (1991), 375--417 

\bibitem{Caf1996}
L.~A. Caffarelli, Allocation maps with general cost functions, in: Partial
  differential equations and applications, Vol. 177 of Lecture Notes in Pure
  and Appl. Math., Dekker, New York, 1996, pp. 29--35.


\bibitem{ChaDePJuu2008}
T.~Champion, L.~De~Pascale, P.~Juutinen, The {$\infty$}-{W}asserstein
  distance: local solutions and existence of optimal transport maps, SIAM J.
  Math. Anal., 40~(1) (2008), 1--20

\bibitem{ChaDePJEMS}
T.~Champion, L.~De~Pascale, 
The Monge problem for strictly convex norms in {$\Bbb R^d$}, J. Eur. Math.
 Soc. (JEMS), 12~(6) (2010), 1355--1369

\bibitem{ChaDePDuke}
T.~Champion, L.~De~Pascale, The Monge problem in {$\Bbb R^d$}, Duke Math. J., 157~(3) (2011), 551--572

\bibitem{ChaDeP2014}
T.~Champion, L.~De~Pascale, On the twist condition and {$c$}-monotone transport plans, Discrete Contin. Dyn. Syst., 34 (4),
    (2014), 1339--1353
 
 \bibitem{DiMLou}  
 S.~Di Marino, J.~Louet, 
 The entropic regularization of the Monge problem on the real line, arxiv:\!1703.10457 (2017)

\bibitem{FatFig2010}
A.~Fathi, A.~Figalli, Optimal transportation on
  non-compact manifolds, Israel J. Math., 175 (2010), 1--59

\bibitem{GanMcC1995}
W.~Gangbo, R.~J. McCann, Optimal maps in Monge's mass transport problem, C.
  R. Acad. Sci. Paris S\'er. I Math., 321~(12) (1995), 1653--1658

\bibitem{GanMcC1996}
W.~Gangbo, R.~J. McCann, The geometry of optimal transportation, Acta Math., 177~(2) (1996), 113--161

\bibitem{JimSan} C.~Jimenez, F.~Santambrogio, Optimal transportation for a quadratic cost with convex constraints 
and applications, J. Math. Pures Appl., 98~(1) (2012),  103–-113 

\bibitem{Jylha} H.~Jylh{\"a}, The $L^\infty$ optimal transport: infinite cyclical monotonicity and the existence of optimal transport maps, Calc. Var. PDEs, (5)~ 1 (2015), 303--326

\bibitem{LouPraZei} J.~Louet, A.~Pratelli, F.~Zeisler, On the continuity of the total cost in the mass transport problem with relativistic cost functions,
  arXiv\! :1612.06229

\bibitem{SanOTAM} F.~Santambrogio, Optimal transport for applied mathematicians, Birkh\"auser, 2015
  
\bibitem{Vil}
C.~Villani, Topics in optimal transportation, Vol.~58 of Graduate Studies in Mathematics, American
  Mathematical Society, Providence, RI, 2003.

\bibitem{Vil2009}
C.~Villani, Optimal transport, Vol. 338 of Grundlehren der Mathematischen Wissenschaften
  [Fundamental Principles of Mathematical Sciences], Springer-Verlag, Berlin, 2009.



\end{thebibliography}
\end{document}